\documentclass[a4paper]{article}
\usepackage[T1]{fontenc}
\usepackage[utf8]{inputenc}
\usepackage[margin=27mm]{geometry}
\usepackage{amsmath,amssymb,amsthm}
\usepackage{mathtools}
\usepackage{microtype}
\usepackage{xcolor}
\usepackage[colorlinks=true,linkcolor=blue,citecolor=blue,urlcolor=blue]{hyperref}
\hypersetup{pdftitle={Two-term small-time spectral expansions for controllability Gramians},pdfauthor={Emmanuel Trelat and Enrique Zuazua}}

\def\R{\mathrm{I\kern-0.21emR}}
\newcommand{\C}{\mathbb{C}}

\renewcommand{\geq}{\geqslant}
\renewcommand{\leq}{\leqslant}

\newtheorem{theorem}{Theorem}
\newtheorem{lemma}[theorem]{Lemma}
\newtheorem{corollary}[theorem]{Corollary}
\newtheorem{proposition}[theorem]{Proposition}
\theoremstyle{definition}
\newtheorem{example}[theorem]{Example}
\theoremstyle{remark}
\newtheorem{remark}[theorem]{Remark}

\newcommand{\Id}{\mathrm{Id}}
\newcommand{\tr}{\operatorname{tr}}
\newcommand{\Span}{\operatorname{span}}
\newcommand{\dist}{\operatorname{dist}}
\newcommand{\calK}{\mathcal{K}}
\newcommand{\calP}{\mathcal{P}}
\newcommand{\calC}{\mathcal{C}}
\newcommand{\calL}{\mathcal{L}}
\newcommand{\calE}{\mathcal{E}}
\newcommand{\calF}{\mathcal{F}}
\newcommand{\diag}{\operatorname{diag}}

\title{Two-term small-time spectral expansions for controllability Gramians}

\author{
Emmanuel Tr\'elat\thanks{Sorbonne Universit\'e, Universit\'e Paris Cit\'e, CNRS, Inria, Laboratoire Jacques-Louis Lions (LJLL), F-75005 Paris, France (\texttt{emmanuel.trelat@sorbonne-universite.fr}).}
\and
Enrique Zuazua\thanks{Chair for Dynamics, Control, Machine Learning and Numerics (AvH Professorship), Department of Mathematics, FAU Erlangen-N\"urnberg, 91058 Erlangen, Germany; Chair of Computational Mathematics, Fundaci\'on Deusto, Bilbao, Spain; Departamento de Matem\'aticas, Universidad Aut\'onoma de Madrid, Spain (\texttt{enrique.zuazua@fau.de}).}
}

\date{}

\begin{document}

\maketitle

\begin{abstract}
For a controllable single-input linear system in dimension $n$, the controllability Gramian eigenvalues, ordered decreasingly, are known to have the successive small-time orders $T, T^3, \ldots, T^{2n-1}$. We refine this leading-order hierarchy by explicitly computing the next-order coefficient of every eigenvalue and the first-order variation of its eigendirection and of the nested spectral subspaces. The resulting formulas have intrinsic expressions in the orthonormal Krylov basis: eigenvalue corrections describe the action of the dynamics along each Krylov direction, while variations of the spectral subspaces describe coupling to the next direction. We also establish local joint real-analytic dependence on $(T, A, b)$ near $T=0$ for the eigenvalues divided by their leading powers of $T$, consistently oriented eigenvectors, and spectral projectors. This provides convergent expansions with remainders uniform on compact families of controllable pairs. These results yield refined asymptotics for the worst-case minimum energy required to reach a unit target from the origin, the Gramian determinant and condition number, and the Ornstein--Uhlenbeck Gaussian profile in moving principal coordinates. The energy equals the reciprocal of the smallest Gramian eigenvalue, whose eigendirection identifies the most energy-demanding targets for sufficiently small times. We further derive an exact Gramian-weighted energy identity for the forward Fokker--Planck equation and sharp anisotropic short-time smoothing asymptotics using, respectively, the Lyapunov equation and an exact Fourier norm formula combined with the graded Gramian factorization. For symmetric dynamics, the spectral data reconstruct $A$ and determine $b$ up to its global sign.
\end{abstract}



\section{Introduction and main result}

\subsection{Problem setting and motivation}\label{intro_setting}

Let $A\in\R^{n\times n}$ and $b\in\R^n$. We consider the single-input linear autonomous control system
$\dot x(t)=Ax(t)+bu(t)$
and assume throughout the spectral analysis that $(A,\allowbreak b)$ satisfies the Kalman rank condition
$$
\operatorname{rank}(b,Ab,\ldots,A^{n-1}b)=n,
\qquad\text{equivalently,}\qquad
\Span(b,Ab,\ldots,A^{n-1}b)=\R^n.
$$
We denote by $\calC$ the open set of controllable pairs. Equivalently, $b$ is a cyclic vector for $A$; see, e.g., \cite{Kailath1980}. We use the Euclidean scalar product $\langle\cdot,\allowbreak \cdot\rangle$ and its associated norm $\Vert\cdot\Vert$ throughout.

For $T>0$, the controllability Gramian is
\begin{equation}\label{gramian}
G_T(A,b)=\int_0^T e^{tA}bb^\top e^{tA^\top}\,dt.
\end{equation}
Equivalently, it solves the differential Lyapunov equation $\partial_TG_T=AG_T+G_TA^\top+bb^\top$ with $G_0=0$.
The Kalman condition ensures that $G_T(A,\allowbreak b)$ is positive definite for every $T>0$. We write its eigenvalues in decreasing order as
$\lambda_1(T)\geq\cdots\geq\lambda_n(T)>0$.
The minimum energy required to steer the origin to $y\in\R^n$ in time $T$ is
$$
\inf_{\substack{u\in L^2(0,T)\\x(0)=0,\ x(T)=y}}
\int_0^T\vert u(t)\vert^2\,dt
=\langle G_T(A,b)^{-1}y,y\rangle,
$$
and its worst value over unit targets is
$$
\sup_{\Vert y\Vert=1}\langle G_T(A,b)^{-1}y,y\rangle
=\frac{1}{\lambda_n(T)};
$$
see, e.g., \cite{MicuZuazua2004,Trelat2024,TucsnakWeiss2009}. Thus the reciprocal of the smallest eigenvalue gives the worst-case minimum control energy, and every unit vector in its eigenspace is a least controllable direction.

The same matrix is the covariance at time $T$ of the degenerate Ornstein--Uhlenbeck process driven by a real Brownian motion,
$dX_t=AX_t\,\allowbreak dt+b\,\allowbreak dW_t,\allowbreak \qquad X_0=0$.
The Kalman condition is then the finite-dimensional form of the H\"ormander condition, and the eigenvalues and eigenvectors of $G_T(A,\allowbreak b)$ are respectively the squared semiaxes and the principal directions of the covariance ellipsoid; see \cite{Habermann2019,Hormander1967,LanconelliPolidoro1994}.

The successive small-time spectral orders are known:
$$
\lambda_k(T)\asymp T^{2k-1},\qquad k\in\{1,\ldots,n\}.
$$
Fast-control asymptotics describe the last layer and limiting norms on the Kalman filtration \cite{Seidman1988,SeidmanYong1997}, while all spectral powers also appear in motion planning \cite{SchmerlingJansonPavone2015}. These leading orders describe the anisotropy of small-time controllability, but do not specify the first drift-dependent correction to each eigenvalue or the first variation of its eigendirection.

We compute explicit intrinsic formulas for the first relative correction to every eigenvalue and the first variation of its eigendirection and of the nested spectral subspaces. The normalized eigenvalues, oriented eigenvectors and flag projectors depend jointly real analytically on $(T,\allowbreak A,\allowbreak b)$ near $T=0$, with convergent expansions and remainders uniform on compact controllable families. The leading exponents, moment constants and limiting Krylov flag are known; the contribution is the simultaneous correction formulas with their analytic and uniform parameter dependence. We compare these results with the literature below.

\subsection{Krylov coordinates and notation}\label{intro_notation}

To state the result, introduce the Krylov filtration
$$
\calK_k(A,b)=\Span(b,Ab,\ldots,A^{k-1}b),\qquad
k\in\{1,\ldots,n\},\qquad \calK_0(A,b)=\{0\}.
$$
Under the Kalman condition, $\dim\calK_k(A,\allowbreak b)=k$, so each layer introduces exactly one new direction. This is a feature of the controllable single-input setting considered here. With several inputs, spectral clusters occur within the quotient layers of the Kalman filtration, leading to a different perturbation problem, discussed among the perspectives at the end of the paper.

For vectors $v_1,\allowbreak \ldots,\allowbreak v_k\in\R^n$, we write
$$
\operatorname{vol}_k(v_1,\ldots,v_k)=\det(\langle v_i,v_j\rangle)_{1\leq i,j\leq k}^{1/2},
$$
and set $\operatorname{vol}_0=1$. Applying Gram-Schmidt to $b,\allowbreak Ab,\allowbreak \ldots,\allowbreak A^{n-1}b$ in this order gives a unique orthonormal basis $(q_1,\allowbreak \ldots,\allowbreak q_n)$ with positive residual coefficients, characterized by
$$
q_k\in\calK_k(A,b),\qquad q_k\perp\calK_{k-1}(A,b),\qquad d_k(A,b)=\langle q_k,A^{k-1}b\rangle>0.
$$
We write $Q(A,\allowbreak b)=(q_1,\allowbreak \ldots,\allowbreak q_n)$ for the resulting orthogonal matrix. Equivalently, $d_1(A,\allowbreak b)=\Vert b\Vert$ and, for $k\in\{2,\allowbreak \ldots,\allowbreak n\}$,
$$
d_k(A,b)=\dist(A^{k-1}b,\calK_{k-1}(A,b))=\frac{\operatorname{vol}_k(b,Ab,\ldots,A^{k-1}b)}{\operatorname{vol}_{k-1}(b,Ab,\ldots,A^{k-2}b)}.
$$
We also introduce
$a_k=\langle q_k,\allowbreak Aq_k\rangle,\allowbreak \qquad k\in\{1,\allowbreak \ldots,\allowbreak n\}$,
and
$$
\beta_k=\langle q_{k+1},Aq_k\rangle=\frac{d_{k+1}(A,b)}{d_k(A,b)}>0,\qquad k\in\{1,\ldots,n-1\},
$$
with the conventions $\beta_0=\beta_n=0$. The last identity follows from $A\calK_k(A,\allowbreak b)\subset\calK_{k+1}(A,\allowbreak b)$. Finally, define the universal constants
\begin{equation}\label{kappa}
\kappa_k=\frac{((k-1)!)^2}{(2k-1)!(2k-2)!},\qquad k\in\{1,\ldots,n\}.
\end{equation}

\subsection{Main result and immediate consequences}\label{intro_results}

With this notation, the main result is as follows.

\begin{theorem}[two-term spectral expansion]\label{thm_main}
Assume that $(A,\allowbreak b)$ satisfies the Kalman condition. There exists $T_0>0$ such that, for every $T\in(0,\allowbreak T_0)$, the eigenvalues satisfy $\lambda_1(T)>\cdots>\lambda_n(T)>0$ and, for every $k\in\{1,\allowbreak \ldots,\allowbreak n\}$, the eigenspace associated with $\lambda_k(T)$ is not orthogonal to $q_k$. On this interval, let $\psi_k(T)$ be the unit eigenvector associated with $\lambda_k(T)$ and oriented by $\langle \psi_k(T),\allowbreak q_k\rangle>0$, and let $\Pi_k(T)$ be the orthogonal projector onto $\Span(\psi_1(T),\allowbreak \ldots,\allowbreak \psi_k(T))$, with $\Pi_0(T)=0$ and $\Pi_n(T)=\Id$.

For every $k\in\{1,\allowbreak \ldots,\allowbreak n\}$, the function $T^{-(2k-1)}\lambda_k(T)$ extends to a real-analytic function on a neighborhood of $T=0$, and
\begin{equation}\label{eigenvalue_expansion}
\lambda_k(T) = \kappa_k\, d_k(A,b)^2\, T^{2k-1}\, (1+a_kT+\mathrm{O}(T^2)),
\end{equation}
as $T\to 0^+$.
For every $k\in\{1,\allowbreak \ldots,\allowbreak n-1\}$, $\Pi_k(T)$ extends real analytically to a neighborhood of $T=0$ and
\begin{equation}\label{projector_expansion}
\Pi_k(T)=\Pi_{\calK_k(A,b)}+\frac{T}{2}\beta_k(q_{k+1}q_k^\top+q_kq_{k+1}^\top)+\mathrm{O}(T^2),
\end{equation}
as $T\to 0^+$.
The eigenvectors may be chosen real analytically on a neighborhood of $T=0$ and satisfy
\begin{equation}\label{eigenvector_expansion}
\psi_k(T)=q_k+\frac{T}{2}(\beta_kq_{k+1}-\beta_{k-1}q_{k-1})+\mathrm{O}(T^2),
\end{equation}
as $T\to 0^+$, where the nonexistent terms for $k=1$ and $k=n$ are omitted. More precisely, near every controllable pair $(A_0,\allowbreak b_0)$, the normalized eigenvalues, the spectral projectors and the oriented eigenvectors extend jointly real analytically in $(T,\allowbreak A,\allowbreak b)$ to a neighborhood of $(0,\allowbreak A_0,\allowbreak b_0)$. For a given compact $K\subset\calC$, there exist $T_K>0$ and an open neighborhood $U_K\subset\calC$ of $K$ such that these extensions are jointly real analytic on $(-T_K,\allowbreak T_K)\times U_K$ and, for every $(A,\allowbreak b)\in K$, their Taylor series in $T$ at zero converge for $\vert T\vert<T_K$. After decreasing $T_K$ if necessary, the eigenvalues are simple and strictly ordered for every $(A,\allowbreak b)\in K$ and every $T\in(0,\allowbreak T_K)$, and all the displayed remainders are uniform on $K$.
\end{theorem}

In \eqref{eigenvalue_expansion}, the leading coefficient is the product of the universal moment constant $\kappa_k$ and the squared Krylov height $d_k(A,\allowbreak b)^2$; the relative correction $a_k$ measures the component of $Aq_k$ along $q_k$. Equation~\eqref{projector_expansion} gives the derivative of the $k$-th flag projector as $\frac12\beta_k(q_{k+1}q_k^\top+q_kq_{k+1}^\top)$: the leading spectral subspace moves only towards the next Krylov direction. Thus the formulas distinguish the first changes in squared principal semiaxes and principal directions.

The control energy, determinant and condition number follow directly.

\begin{corollary}[control energy, determinant and conditioning]
Under the assumptions of Theorem~\ref{thm_main}, the worst-case minimum control energy satisfies
\begin{equation}\label{control_energy}
\sup_{\Vert y\Vert=1}\inf_{\substack{u\in L^2(0,T)\\x(0)=0,\ x(T)=y}}\int_0^T\vert u(t)\vert^2dt=\frac{T^{-(2n-1)}}{\kappa_nd_n(A,b)^2}(1-a_nT+\mathrm{O}(T^2)).
\end{equation}
Moreover,
\begin{equation}\label{determinant_expansion}
\det G_T(A,b)=(\det K_b)^2\det\mathfrak H_nT^{n^2}(1+\tr(A)T+\mathrm{O}(T^2)),
\end{equation}
where $K_b=(b,\allowbreak Ab,\allowbreak \ldots,\allowbreak A^{n-1}b)$ and $\mathfrak H_n=(1/((i-1)!(j-1)!(i+j-1)))_{1\leq i,\allowbreak j\leq n}$ is the scaled Hilbert moment matrix. Finally,
$$
\frac{\lambda_1(T)}{\lambda_n(T)}=\frac{d_1(A,b)^2}{\kappa_nd_n(A,b)^2}T^{-2(n-1)}(1+(a_1-a_n)T+\mathrm{O}(T^2)).
$$
\end{corollary}

\begin{proof}
The first and third formulas follow by inversion and division in Theorem~\ref{thm_main}. Multiplying \eqref{eigenvalue_expansion} over $k$ and using Gram-Schmidt and the moment determinant formula gives
$$
\prod_{k=1}^nd_k(A,b)=\vert\det K_b\vert,\qquad \prod_{k=1}^n\kappa_k=\det\mathfrak H_n,
$$
while $\sum_{k=1}^na_k=\tr(Q^\top AQ)=\tr A$. This proves \eqref{determinant_expansion}.
\end{proof}

Formula \eqref{control_energy} quantifies the first drift-dependent correction to the classical blow-up of the fast-control cost. Formula \eqref{determinant_expansion} is included to show how the individual coefficients recombine: since $\sum_ka_k=\tr A$, it recovers the aggregate correction already known from \cite{BarilariPaoli2017}.

\begin{example}[the double integrator]\label{ex_integrator}
Let
$$
A=\begin{pmatrix}0&1\\0&0\end{pmatrix},\qquad b=\begin{pmatrix}0\\1\end{pmatrix}.
$$
Then $q_1=e_2$, $q_2=e_1$, $d_1(A,\allowbreak b)=d_2(A,\allowbreak b)=1$, $a_1=a_2=0$ and $\beta_1=1$. The Gramian is
$$
G_T(A,b)=\begin{pmatrix}T^3/3&T^2/2\\T^2/2&T\end{pmatrix}.
$$
Consequently,
$$
\lambda_1(T)=T(1+T^2/4+\mathrm{O}(T^4)),\qquad \lambda_2(T)=\frac{T^3}{12}(1-T^2/4+\mathrm{O}(T^4)),
$$
whereas
$$
\psi_1(T)=e_2+\frac{T}{2}e_1+\mathrm{O}(T^2),\qquad \psi_2(T)=e_1-\frac{T}{2}e_2+\mathrm{O}(T^2),
$$
as $T\to 0^+$.
Thus the relative term of order $T$ vanishes for both eigenvalues, but the principal directions already vary at order $T$.
\end{example}

\begin{remark}[Local spectral analyticity]
The Gramian is entire in $T$, whereas Theorem~\ref{thm_main} establishes analyticity of its normalized spectral data only near zero. No quantitative lower bound for $T_K$ is obtained. Degenerating Krylov heights and complex spectral collisions can obstruct parameter-uniform bounds. Even for a fixed pair, the theorem does not exclude eigenvalue collisions at positive times, where ordered rank-one projectors cease to be intrinsic. In the absence of real collisions, complex collisions can still limit the Taylor radius: continuation of a separated projector along a real interval does not imply convergence there of its Taylor series at zero.
\end{remark}

\begin{remark}[complex extension]
Theorem~\ref{thm_main}, Proposition~\ref{prop_centered} below and their proofs extend to controllable pairs $(A,\allowbreak b)\in\C^{n\times n}\times\C^n$. The Gramian is then the Hermitian matrix $G_T(A,\allowbreak b)=\int_0^Te^{tA}bb^*e^{tA^*}dt$; transposes are replaced by adjoints and the Krylov frame is unitary. The normalized eigenvalues, projectors and phase-oriented eigenvectors are real analytic with respect to the real and imaginary parts of the parameters. Fixing the phases by requiring $q_k^*\psi_k(T)>0$ real, the projector and eigenvector formulas are unchanged. The only modification in the eigenvalue formula is $a_k=\operatorname{Re}(q_k^*Aq_k)$, whereas $\beta_k=q_{k+1}^*Aq_k=d_{k+1}(A,\allowbreak b)/d_k(A,\allowbreak b)>0$ remains real and positive by construction. In particular,
$$
\det G_T(A,b)=\vert\det K_b\vert^2\det\mathfrak H_nT^{n^2}(1+\operatorname{Re}(\tr A)T+\mathrm{O}(T^2)).
$$
The graded factorization, the parity argument on the centered interval and the differentiation of Gram determinants apply in the Hermitian setting without further change. This complex spectral version is used in \cite{TrelatZuazua2026actuator} for Schr\"odinger truncations.
\end{remark}

\subsection{Relation to previous work}\label{intro_related}

Beyond the fast-control and motion-planning results cited above, several works describe the leading-order behavior and geometry of controllability Gramians. Small-parameter estimates for the least eigenvalue are developed in \cite{Gusev2019,Gusev2020}, while \cite{Osipov2025} gives a recursive expansion of the Gramian and discusses the resulting eigenvalue asymptotics for second-order systems. A related limit-shape result, with a rate of convergence, is proved for small-time reachable sets under integral control constraints in \cite{GoncharovaOvseevich2016}. Adapted dilations and limiting covariances occur for hypoelliptic diffusions \cite{Habermann2019,LanconelliPolidoro1994}, and the first aggregate correction of the covariance determinant is obtained in \cite{BarilariPaoli2017}. Related near-diagonal hypoelliptic heat-kernel expansions for broader classes of operators are developed in \cite{ColinHillairetTrelat2021,Paoli2017}.

On the spectral-perturbation side, eigenvalue scaling in graded matrices, under conditions ensuring that the spectrum inherits the grading, was analyzed in \cite{StewartZhang1990}. For one-dimensional smooth kernel matrices in the flat limit, QR factorizations and orthogonal polynomials determine limiting eigenspaces \cite{BarthelmeUsevich2021}. A general framework based on successive Schur complements for the leading eigenvalues and limiting eigenspaces of graded symmetric matrices is developed in \cite{UsevichBarthelme2026}. Finally, symmetric-window covariances of regular curves, including parity cancellations and limiting Frenet directions, are studied in \cite{AlvarezVizosoEtAl2019,Solis2000}. Next-to-leading eigenvalue terms and limiting eigendirections have also been computed for local covariance operators of embedded manifolds in \cite{AlvarezVizosoKirbyPeterson2020}. That problem is geometrically distinct: the covariance is formed from a symmetric manifold patch and its coefficients encode curvature tensors, whereas here the singular filtration is generated by a one-sided controllability orbit and the linear term is produced by its transport. The determinant formula behind \eqref{kappa} is the classical Hilbert determinant formula; see \cite{Choi1983,Szego1975}.

The coefficient $\beta_k$ also has a direct hypoelliptic interpretation. For the Ornstein-Uhlenbeck generator $\calL=(Ax)\cdot\nabla+\frac{1}{2}(b\cdot\nabla)^2$, set $Z=(Ax)\cdot\nabla$ and $Y_k=q_k\cdot\nabla$. Since $Aq_k\in\calK_{k+1}(A,\allowbreak b)$,
$$
[Y_k,Z]=\beta_kY_{k+1}+\sum_{j=1}^k\langle q_j,Aq_k\rangle Y_j,\qquad k\in\{1,\ldots,n-1\}.
$$
Thus $a_k$ is the coefficient acting within the $k$-th derivative layer, while $\beta_k$ is the coefficient with which commutation by the drift creates the next Krylov derivative. The corresponding short-time derivative costs $t^{-k+1/2}$ are classical, in the sup-norm framework for an arbitrary drift \cite{Lunardi1997}, and in several $L^2$ frameworks for degenerate Ornstein-Uhlenbeck, Fokker-Planck and quadratic semigroups \cite{AlphonseBernier2020,FarkasLunardi2006,Herau2007,HitrikPravdaStarovViola2017}. We use these estimates only as context. What Theorem~\ref{thm_main} adds is that $a_k$ and $\beta_k$ govern the covariance itself: $a_k$ corrects the $k$-th layer, while $\beta_k$ produces the first rotation of the spectral flag. Proposition~\ref{prop_energy} below gives a complementary exact identity for the adjoint forward equation, for which the drift commutators have the opposite sign.

In particular, the smallest eigenvalue comes from the last Krylov layer, relevant to actuator design, and lies on the scale $T^{2n-1}$. Finite-dimensional scalar actuator design was previously approached through Brunovsk\'y coordinates in \cite{GeshkovskiZuazua2022}. As clarified in \cite{GeshkovskiZuazuaCorrigendum}, the underlying Gramian congruence is exact, but the asserted factorization of the control cost is generally invalid; the resulting horizon-independent optimization remains valid as a surrogate upper-bound problem. The present paper develops the fixed-pair spectral theory independently. The companion preprint \cite{TrelatZuazua2026actuator} applies these expansions to optimization over actuator covariances, sharp small-time rank losses and phase-aligned recovery across control horizons.

Related graded phenomena arise in nonlinear and bilinear small-time controllability: moment problems can yield local controllability in every positive time \cite{BeauchardLaurent2010}, while iterated Lie brackets may lead to quadratic obstructions when the linearization is not controllable \cite{BeauchardMarbach2018}. The present fully controllable finite-dimensional setting is distinct, but gives an explicit spectral realization of an analogous hierarchy.

To our knowledge, the simultaneous two-term Euclidean spectral formulas of Theorem~\ref{thm_main}, together with their joint analytic dependence and compact-family uniformity, do not appear in the cited literature. Besides the leading exponents, moment constants and limiting Krylov flag already mentioned, the centered-interval technique and the first trace correction of the determinant are not claimed as new. Nor do we claim that the next coefficients are inaccessible to general graded perturbation schemes. Here, the moment structure and midpoint parity yield explicit intrinsic formulas for all the first corrections.

\subsection{Proof strategy and organization of the paper}\label{intro_strategy}

The proof starts from the entire Gramian expansion
$$
G_T(A,b)=\sum_{i,j\geq0}\frac{T^{i+j+1}}{i!j!(i+j+1)}A^ibb^\top(A^\top)^j.
$$
The factors $1/(i+j+1)=\int_0^1s^{i+j}\,\allowbreak ds$, weighted by $1/(i!j!)$, form the scaled Hilbert matrix $\mathfrak H_n$. Equivalently, if $G_T=\sum_{m\geq1}T^mC_m$, then $C_1=bb^\top$ and $C_{m+1}=(AC_m+C_mA^\top)/(m+1)$ for $m\geq1$. These formulas determine the matrix Taylor coefficients, but the rank-one leading term only gives $\lambda_1(T)=T\Vert b\Vert^2+\mathrm{O}(T^2)$; resolving the other branches requires the graded analysis.

The component of $e^{tA}b$ along $q_k$ first appears at order $t^{k-1}$, so the $k$-th Krylov diagonal entry of $G_T$ has order $T^{2k-1}$. Lemma~\ref{lem_graded} identifies the positive definite limiting moment matrix and establishes the spectral scales and joint analyticity. Centering the time interval removes the first odd correction; transporting the Krylov frame then produces $a_k$ and $\beta_k$. Midpoint translation is used only as a proof device.

Section~\ref{sec_proof} proves the theorem. Section~\ref{sec_consequences} derives a refined Gaussian profile, records an even expansion for skew-symmetric dynamics, obtains an inverse reconstruction for symmetric dynamics, and finally establishes an exact Gramian-weighted energy identity and sharp directional smoothing estimates for the forward Fokker-Planck equation. Section~\ref{sec_conclusion} delimits the scope of the result: higher-order coefficients, addition of a scalar drift, several inputs and growing dimension.

\section{Proof of the main result}\label{sec_proof}

We prove Theorem~\ref{thm_main} through graded factorization, centering and transport of the Krylov frame.

\subsection{A graded factorization}

The first step establishes the leading orders, simplicity and joint analyticity.

Recall that $\calC$ is the open set of pairs $(A,\allowbreak b)\in\R^{n\times n}\times\R^n$ satisfying the Kalman condition. On $\calC$ the Gram-Schmidt frame $Q=Q(A,\allowbreak b)$ and the heights $d_k(A,\allowbreak b)$ are real analytic. For real $T$, set $\Delta(T)=\diag(1,\allowbreak T,\allowbreak \ldots,\allowbreak T^{n-1})$. For $T>0$, also set
$\Lambda(T)=\diag(T^{1/2},\allowbreak T^{3/2},\allowbreak \ldots,\allowbreak T^{n-1/2})=T^{1/2}\Delta(T)$.
For $k\in\{1,\allowbreak \ldots,\allowbreak n\}$, recall that $\mathfrak H_k=(1/((i-1)!(j-1)!(i+j-1)))_{1\leq i,\allowbreak j\leq k}$ is the scaled Hilbert moment matrix of \eqref{determinant_expansion}, with $\det\mathfrak H_0=1$, so that $\det\mathfrak H_k/\det\mathfrak H_{k-1}=\kappa_k$ by the Hilbert determinant formula.

Hereafter, we use standard exterior-power notation. If $I=\{i_1<\cdots<i_k\}\subset\{1,\allowbreak \ldots,\allowbreak n\}$, set $e_I=e_{i_1}\wedge\allowbreak \cdots\wedge\allowbreak  e_{i_k}$. For a matrix $M$, the induced operator on ${\bigwedge}^k\R^n$ is denoted by ${\bigwedge}^kM$; its $(I,\allowbreak J)$ entry in the basis $(e_I)$ is the minor $\det M_{I,\allowbreak J}$, and its eigenvalues are the products of $k$ eigenvalues of $M$. We write $M_{I,\allowbreak I}$ for the principal submatrix indexed by $I$.

\begin{lemma}[graded factorization]\label{lem_graded}
Extend $G_T(A,\allowbreak b)$ to $T\in\R$ by the same integral as in \eqref{gramian}. There exists a symmetric matrix $N(T,\allowbreak A,\allowbreak b)$, real analytic in a neighborhood of $\{0\}\times\calC$, such that, for every real $T$ near zero,
\begin{equation}\label{graded_factorization}
Q^\top G_T(A,b)Q=T\Delta(T)N(T,A,b)\Delta(T).
\end{equation}
Equivalently, $(Q^\top G_T(A,\allowbreak b)Q)_{ij}=T^{i+j-1}N_{ij}(T,\allowbreak A,\allowbreak b)$. For $T>0$, formula \eqref{graded_factorization} can also be written $Q^\top G_T(A,\allowbreak b)Q=\Lambda(T)N(T,\allowbreak A,\allowbreak b)\Lambda(T)$. Moreover,
$$
N(0,A,b)=\big(\frac{d_i(A,b)d_j(A,b)}{(i-1)!(j-1)!(i+j-1)}\big)_{1\leq i,j\leq n}\succ0.
$$
Equivalently, with $D_d(A,\allowbreak b)=\diag(d_1(A,\allowbreak b),\allowbreak \ldots,\allowbreak d_n(A,\allowbreak b))$,
\begin{equation}\label{moment_factorization}
N(0,A,b)=D_d(A,b)\mathfrak H_nD_d(A,b).
\end{equation}
Consequently, near every point of $\{0\}\times\calC$ and for every $k\in\{1,\allowbreak \ldots,\allowbreak n\}$, the normalized product $T^{-k^2}\prod_{j=1}^k\lambda_j(T)$ extends jointly real analytically in $(T,\allowbreak A,\allowbreak b)$. So do $T^{-(2k-1)}\lambda_k(T)$, the flag projector $\Pi_k(T)$ and the oriented eigenvector $\psi_k(T)$, with
\begin{equation}\label{leading_order}
\begin{aligned}
\lambda_1(T)\cdots\lambda_k(T)&=\Big(\prod_{j=1}^kd_j(A,b)^2\Big)\det\mathfrak H_kT^{k^2}(1+\mathrm{O}(T)),\\
\Pi_k(T)&=\Pi_{\calK_k(A,b)}+\mathrm{O}(T),
\end{aligned}
\end{equation}
as $T\to 0^+$.
In particular, $\lambda_k(T)=\kappa_kd_k(A,\allowbreak b)^2T^{2k-1}(1+\mathrm{O}(T))$, and the eigenvalues are simple for all sufficiently small $T>0$. Finally,
\begin{equation}\label{minor_identity}
\lambda_1(T)\cdots\lambda_k(T)=\det((Q^\top G_T(A,b)Q)_{I_k,I_k})(1+\mathrm{O}(T^2)),
\end{equation}
where $I_k=\{1,\allowbreak \ldots,\allowbreak k\}$. Thus the determinant on the right-hand side is the leading principal minor of order $k$.
All remainders are locally uniform on $\calC$.
\end{lemma}

\begin{proof}
Set $\varphi_j(T,\allowbreak s)=T^{-(j-1)}\langle q_j,\allowbreak e^{TsA}b\rangle$ for $s\in(0,\allowbreak 1)$. Since $\langle q_j,\allowbreak A^mb\rangle=0$ for $m\leq j-2$, the series
$$
\langle q_j,e^{TsA}b\rangle=\sum_{m\geq j-1}\frac{(Ts)^m}{m!}\langle q_j,A^mb\rangle
$$
is divisible by $T^{j-1}$, so $\varphi_j$ is real analytic in $(T,\allowbreak A,\allowbreak b)$ with values in $L^2(0,\allowbreak 1)$ and $\varphi_j(0,\allowbreak \cdot)=d_js^{j-1}/(j-1)!$. The change of variable $t=Ts$, which is valid for signed $T$, gives
$(Q^\top G_TQ)_{ij}=T^{i+j-1}\langle \varphi_i,\allowbreak \varphi_j\rangle_{L^2(0,\allowbreak 1)}$.
This is \eqref{graded_factorization} with $N_{ij}=\langle \varphi_i,\allowbreak \varphi_j\rangle_{L^2(0,\allowbreak 1)}$. The matrix $N(0,\allowbreak A,\allowbreak b)$ is the Gram matrix of the linearly independent functions $d_js^{j-1}/(j-1)!$, hence is positive definite, and its entries give \eqref{moment_factorization}.

Fix $k$ and, for $I=\{i_1<\cdots<i_k\}$, let $\delta(I)=\sum_{r=1}^k(i_r-r)$, so that $\delta(I)\geq0$ with equality only for $I=I_k$. Taking minors in the entrywise identity gives, for every pair $(I,\allowbreak J)$, the factor $T^{k^2+\delta(I)+\delta(J)}$. Hence
$$
{\bigwedge}^k(Q^\top G_TQ)=T^{k^2}R(T,A,b),\qquad R=D\,{\bigwedge}^kN\,D,\qquad D(T)=\diag(T^{\delta(I)}).
$$
The family $R$ is symmetric and real analytic, and $R(0,\allowbreak A,\allowbreak b)=\det N(0,\allowbreak A,\allowbreak b)_{I_k,\allowbreak I_k}\,\allowbreak e_{I_k}e_{I_k}^\top$, whose largest eigenvalue $\det N(0,\allowbreak A,\allowbreak b)_{I_k,\allowbreak I_k}>0$ is simple. A simple eigenvalue of a real-analytic symmetric family, together with its spectral projector and a locally oriented unit eigenvector, depends real analytically on all the parameters; see \cite{Kato1995}. For $T>0$, the eigenvalues of the $k$-th exterior power are the products of $k$ eigenvalues of $Q^\top G_TQ$, and the largest one is $\lambda_1(T)\cdots\lambda_k(T)$. Hence this simple eigenvalue branch is $T^{-k^2}\lambda_1(T)\cdots\lambda_k(T)$. This proves the asserted analytic extension, and $\det N(0,\allowbreak A,\allowbreak b)_{I_k,\allowbreak I_k}=(\prod_{j=1}^kd_j^2)\det\mathfrak H_k$ gives the first identity in \eqref{leading_order}. Dividing consecutive levels yields the analyticity of $T^{-(2k-1)}\lambda_k(T)$ and the leading term $\kappa_kd_k^2T^{2k-1}$. Since the exponents $2k-1$ are pairwise distinct, the eigenvalues are simple for small $T>0$.

Let $\omega_k(T)$ be the corresponding analytic unit eigenvector of $R$, oriented by $\langle\omega_k(T),\allowbreak e_{I_k}\rangle>0$. For $T>0$, it equals $(Q^\top\psi_1(T))\wedge\allowbreak \cdots\wedge\allowbreak (Q^\top\psi_k(T))$ up to sign and is therefore decomposable. If $\omega$ is a unit decomposable $k$-vector representing a $k$-plane $E$, then, denoting by $\iota_v$ the contraction by $v$, the matrix $(\langle\iota_{e_i}\omega,\allowbreak \iota_{e_j}\omega\rangle)_{1\leq i,\allowbreak j\leq n}$ equals $\Pi_E$. It is a quadratic function of $\omega$ and is insensitive to its sign. Applied to $E_k(T)=\Span(Q^\top\psi_1(T),\allowbreak \ldots,\allowbreak Q^\top\psi_k(T))$, this gives $\Pi_k(T)=Q\Pi_{E_k(T)}Q^\top$ and proves its analyticity. Since $\omega_k(0)=e_{I_k}$, one obtains $\Pi_k(0)=\Pi_{\calK_k(A,\allowbreak b)}$. The identities $\Pi_{k-1}\Pi_k=\Pi_{k-1}$ and $(\Pi_k-\Pi_{k-1})^2=\Pi_k-\Pi_{k-1}$, valid for $T>0$, persist throughout the analytic neighborhood by the identity principle. Thus $P_k(T)=\Pi_k(T)-\Pi_{k-1}(T)$ is an analytic rank-one projector near zero, and $\psi_k(T)=P_k(T)q_k/\Vert P_k(T)q_k\Vert$ is well defined, analytic and satisfies $\langle\psi_k(T),\allowbreak q_k\rangle>0$. This proves the claimed analyticity of the oriented eigenvector.

Finally, for $k=n$, identity \eqref{minor_identity} is exact. If $k<n$, write $R$ in the decomposition $\R e_{I_k}\oplus e_{I_k}^\perp$ as
$$
R=\begin{pmatrix}a&r^\top\\r&B\end{pmatrix}.
$$
Here $a=\det N(T,\allowbreak A,\allowbreak b)_{I_k,\allowbreak I_k}=T^{-k^2}\det((Q^\top G_TQ)_{I_k,\allowbreak I_k})$ tends to a positive number, while $r=\mathrm{O}(T)$ and $B=\mathrm{O}(T^2)$. If $\mu(T)$ denotes the simple largest eigenvalue of $R(T)$, its Schur complement equation is $\mu=a+r^\top(\mu\Id-B)^{-1}r$, and therefore $\mu=a+\mathrm{O}(T^2)$. This proves \eqref{minor_identity}.

Around each point $(0,\allowbreak A_0,\allowbreak b_0)\in\{0\}\times\calC$, choose a product neighborhood on which all the local power-series representations converge, and shrink its parameter factor if necessary. The local branches agree on overlaps: for each fixed $(A,\allowbreak b)$ in a parameter overlap, they coincide with the ordered spectral data for all sufficiently small $T>0$, and the one-variable identity principle in $T$ therefore makes them agree throughout the common $T$-interval. Select finitely many such parameter neighborhoods covering a compact set $K\subset\calC$, let $U_K$ be their union and let $T_K$ be the minimum of their $T$-radii. This also gives the uniform remainders asserted in Theorem~\ref{thm_main}. The leading expansions, whose coefficients are bounded away from zero on $K$, then give the uniform strict ordering for $T\in(0,\allowbreak T_K)$ after decreasing $T_K$ if necessary.
\end{proof}

Lemma~\ref{lem_graded} already contains the leading order and the convergence to the Krylov flag, both of which are known, as recalled in the introduction. What it does not contain is the coefficient of $T$ in \eqref{leading_order}, which is the object of the rest of this section.

\begin{remark}\label{rem_minors}
Identity \eqref{minor_identity} states that, in the Krylov basis, the leading principal minors of the Gramian compute the partial products of its eigenvalues up to a relative error $\mathrm{O}(T^2)$. For $k=n$ the identity is exact, both sides being $\det G_T(A,\allowbreak b)$. It reduces the eigenvalue part of Theorem~\ref{thm_main} to the differentiation of Gram determinants and also indicates how further coefficients may be generated, at the price of the higher-order perturbation terms discarded here.

To identify the linear coefficient directly, set $A_Q=Q^\top AQ$ and $\tau_j=a_1+\cdots+a_j$. Since $A_Q$ is upper Hessenberg, $Q^\top b=d_1(A,\allowbreak b)e_1$, $(A_Q)_{kk}=a_k$ and $(A_Q)_{k+1,\allowbreak k}=\beta_k$, every path of length $j$ from the index $1$ to the index $j$ consists of $j-1$ moves from $\ell$ to $\ell+1$ and one diagonal step. Hence
$\langle q_j,\allowbreak A^jb\rangle=d_j(A,\allowbreak b)\tau_j,\allowbreak \qquad j\in\{1,\allowbreak \ldots,\allowbreak n\}$.
If $f_i(s)=s^i/i!$, $F_k=f_0\wedge\allowbreak \cdots\wedge\allowbreak  f_{k-1}$ and $F_k^+=f_0\wedge\allowbreak \cdots\wedge\allowbreak  f_{k-2}\wedge\allowbreak  f_k$, with $F_1^+=f_1$, then the functions in the proof of Lemma~\ref{lem_graded} satisfy $\varphi_j/d_j(A,\allowbreak b)=f_{j-1}+T\tau_jf_j+\mathrm{O}(T^2)$. In the first variation of $\varphi_1\wedge\allowbreak \cdots\wedge\allowbreak \varphi_k$, only the last substitution survives, since every earlier one duplicates a factor. Moreover $kF_k^+$ is the image of $F_k$ under the derivation induced by the multiplication by $s$, while the reflection $s\mapsto1-s$ induces a unitary map of $L^2(0,\allowbreak 1)$ which preserves the polynomials of degree at most $k-1$, fixes $F_k$ up to sign, and turns that multiplication into $\Id-s$; hence $2\langle F_k,\allowbreak F_k^+\rangle_{L^2(0,\allowbreak 1)}=\Vert F_k\Vert_{L^2(0,\allowbreak 1)}^2$. It follows that the relative coefficient of the leading principal $k$-minor is $\tau_k$. Formula~\eqref{minor_identity}, followed by division of the products at levels $k$ and $k-1$, gives the coefficient $a_k$ in \eqref{eigenvalue_expansion} directly. Thus centering is not needed for the eigenvalue coefficient; it remains essential for the order $\mathrm{O}(T^2)$ of the projectors in \eqref{centered_projectors}.
\end{remark}

\subsection{Centering and the gain of one order}

For $c\in\R^n$, define the covariance on the symmetric time interval
\begin{equation}\label{centered_gramian}
\widehat{G}_T(A,c)=\int_{-T/2}^{T/2}e^{tA}cc^\top e^{tA^\top}dt.
\end{equation}
Here and below, ``centered'' refers to the time interval, not to subtraction of the mean of the curve $t\mapsto e^{tA}c$.
The change of variable $t=s+T/2$ gives the exact identity
\begin{equation}\label{midpoint_identity}
G_T(A,b)=\widehat{G}_T(A,b_T),\qquad b_T=e^{TA/2}b.
\end{equation}
For fixed $c$, the matrix $T^{-1}\widehat{G}_T(A,\allowbreak c)$ is an even real-analytic function of $T$. The following intermediate result makes the spectral consequence of this parity precise.

\begin{proposition}[centered covariance]\label{prop_centered}
Assume that $(A,\allowbreak c)$ satisfies the Kalman condition. For $T>0$ sufficiently small, let $\widehat\lambda_k(T)$ denote the decreasingly ordered eigenvalues of \eqref{centered_gramian}, let $\widehat\Pi_k(T)$ project onto the first $k$ eigenvectors, and choose the individual unit eigenvectors $\widehat\psi_k(T)$ so that $\langle\widehat\psi_k(T),\allowbreak q_k(A,\allowbreak c)\rangle>0$. Then, as $T\to0^+$,
$$
\widehat\lambda_k(T)=\kappa_kd_k(A,c)^2T^{2k-1}(1+\mathrm{O}(T^2)),
$$
\begin{equation}\label{centered_projectors}
\widehat\Pi_k(T)=\Pi_{\calK_k(A,c)}+\mathrm{O}(T^2),\qquad \widehat\psi_k(T)=q_k(A,c)+\mathrm{O}(T^2).
\end{equation}
After division by $T^{2k-1}$, each eigenvalue branch extends to a real-analytic function of $T^2$ near zero; the rank-one eigenprojectors and the oriented eigenvectors have analogous extensions. The estimates are uniform on compact subsets of the set of controllable pairs.
\end{proposition}

We give a self-contained proof because the order $T^2$ in \eqref{centered_projectors} is essential below. We use the standard exterior power of a Hilbert space: if $f_1,\allowbreak \ldots,\allowbreak f_k\in L^2(-1/2,\allowbreak 1/2)$, then $f_1\wedge\allowbreak \cdots\wedge\allowbreak  f_k$ denotes their alternating tensor product and
$$
\Vert f_1\wedge\cdots\wedge f_k\Vert^2=\det(\langle f_i,f_j\rangle_{L^2})_{1\leq i,j\leq k}.
$$
We shall also use the following standard comparison. If $E$ and $F$ are two $k$-dimensional subspaces and $0\leq\Theta_1\leq\cdots\leq\Theta_k\leq\pi/2$ are their principal angles, defined by $\cos\Theta_i=\sigma_i(U^\top V)$, where $U$ and $V$ are orthonormal frames of $E$ and $F$ and the singular values $\sigma_i$ are ordered decreasingly, then
$$
\Vert\Pi_E-\Pi_F\Vert=\max_i\sin\Theta_i,\qquad \big\Vert\Pi_{{\bigwedge}^kE}-\Pi_{{\bigwedge}^kF}\big\Vert=\Big(1-\prod_{i=1}^k\cos^2\Theta_i\Big)^{1/2}.
$$
Thus these two quantities are comparable, with constants depending only on $k$; see \cite{HornJohnson2013}.

\begin{proof}
Let $(q_1(A,\allowbreak c),\allowbreak \ldots,\allowbreak q_n(A,\allowbreak c))$ be the orthonormal Krylov basis associated with $(A,\allowbreak c)$ and set $M(T)=Q^\top\widehat{G}_T(A,\allowbreak c)Q$, where $Q=(q_1(A,\allowbreak c),\allowbreak \ldots,\allowbreak q_n(A,\allowbreak c))$. For $s\in[-1/2,\allowbreak 1/2]$, set $g_j(T,\allowbreak s)=\langle q_j(A,\allowbreak c),\allowbreak e^{TsA}c\rangle$.
Orthogonality to $\calK_{j-1}(A,\allowbreak c)$ and Taylor's formula give, in $L^2(-1/2,\allowbreak 1/2)$,
\begin{equation}\label{centered_coordinates}
g_j(T,s)=T^{j-1}(d_j(A,c)\frac{s^{j-1}}{(j-1)!}+T\alpha_j\frac{s^j}{j!}+T^2r_j(T,s)),
\end{equation}
where $\alpha_j=\langle q_j(A,\allowbreak c),\allowbreak A^jc\rangle$ and the remainders are uniformly bounded on compact controllable families. Since $M_{ij}(T)=T\langle g_i(T),\allowbreak g_j(T)\rangle_{L^2}$, the principal minors of $M(T)$ are Gram determinants, while arbitrary minors are the corresponding cross-Gram determinants.

Let $\overline f_j(s)=s^j/j!$ on $[-1/2,\allowbreak 1/2]$ and let
$$
\overline H_k=\big(\int_{-1/2}^{1/2}\frac{s^{i+j}}{i!j!}\,ds\big)_{0\leq i,j\leq k-1},\qquad \det\overline H_0=1.
$$
Set $\overline F_k=\overline f_0\wedge\allowbreak \cdots\wedge\allowbreak \overline f_{k-1}$ and $\overline F_k^+=\overline f_0\wedge\allowbreak \cdots\wedge\allowbreak \overline f_{k-2}\wedge\allowbreak \overline f_k$, with $\overline F_1^+=\overline f_1$. In the first-order term of $g_1\wedge\allowbreak \cdots\wedge\allowbreak  g_k$, replacing any factor before the last one duplicates the next monomial and gives zero. Hence \eqref{centered_coordinates} yields
$$
g_1(T)\wedge\cdots\wedge g_k(T)=T^{k(k-1)/2}\Big(\prod_{j=1}^kd_j(A,c)\Big)\big(\overline F_k+T\frac{\alpha_k}{d_k(A,c)}\overline F_k^++\mathrm{O}(T^2)\big).
$$
The two exterior products $\overline F_k$ and $\overline F_k^+$ have opposite parity and are therefore orthogonal. It follows that the leading principal $k$-minor satisfies
$$
\det M(T)_{\{1,\ldots,k\}}=\Big(\prod_{j=1}^kd_j(A,c)^2\Big)\det\overline H_kT^{k^2}(1+\mathrm{O}(T^2)).
$$

The same argument controls the associated eigenspace. For $I=\{i_1<\cdots<i_k\}$, set $g_I=g_{i_1}\wedge\allowbreak \cdots\wedge\allowbreak  g_{i_k}$ and $\delta(I)=\sum_{r=1}^k(i_r-r)$. In the canonical coordinates of $M(T)$, one has $({\bigwedge}^kM(T))_{I,\allowbreak J}=T^k\langle g_I(T),\allowbreak g_J(T)\rangle$, and this coefficient is of order $T^{k^2+\delta(I)+\delta(J)}$. When $k<n$, the only off-diagonal coefficient that could have relative order $T$ joins $I_k=\{1,\allowbreak \ldots,\allowbreak k\}$ to $\{1,\allowbreak \ldots,\allowbreak k-1,\allowbreak k+1\}$; its leading scalar product again couples functions of opposite parity and vanishes. Therefore, after identifying the canonical basis with the Krylov basis through $Q$, in the decomposition generated first by $q_1(A,\allowbreak c)\wedge\allowbreak \cdots\wedge\allowbreak  q_k(A,\allowbreak c)$ and then by its orthogonal complement,
\begin{equation}\label{exterior_block}
{\bigwedge}^kM(T)=C_k(A,c)T^{k^2}\begin{pmatrix}1+\mathrm{O}(T^2)&\mathrm{O}(T^2)\\\mathrm{O}(T^2)&\mathrm{O}(T^2)\end{pmatrix},
\end{equation}
where $C_k(A,\allowbreak c)=\Big(\prod_{j=1}^kd_j(A,\allowbreak c)^2\Big)\det\overline H_k$. The largest eigenvalue of \eqref{exterior_block} is $C_k(A,\allowbreak c)T^{k^2}(1+\mathrm{O}(T^2))$ and its eigenline is at distance $\mathrm{O}(T^2)$ from $q_1(A,\allowbreak c)\wedge\allowbreak \cdots\wedge\allowbreak  q_k(A,\allowbreak c)$. For $T>0$, this eigenvalue is $\widehat\lambda_1(T)\cdots\widehat\lambda_k(T)$. Dividing by the formula at level $k-1$ gives
\begin{equation}\label{centered_ratio}
\widehat\lambda_k(T)=d_k(A,c)^2\frac{\det\overline H_k}{\det\overline H_{k-1}}T^{2k-1}(1+\mathrm{O}(T^2)).
\end{equation}
The comparison just recalled therefore converts the exterior-eigenline estimate into the projector formula in \eqref{centered_projectors}. Subtracting consecutive flag projectors gives the analytic rank-one projector onto the $k$-th eigenline, and applying the oriented normalization used in the proof of Lemma~\ref{lem_graded} yields $\widehat\psi_k(T)=q_k(A,\allowbreak c)+\mathrm{O}(T^2)$.

Translation from $[-1/2,\allowbreak 1/2]$ to $[0,\allowbreak 1]$ acts triangularly on polynomials and preserves the determinant of the moment matrix. The Hilbert determinant formula consequently gives
$$
\frac{\det\overline H_k}{\det\overline H_{k-1}}=\dist_{L^2(0,1)}\big(\frac{t^{k-1}}{(k-1)!},\calP_{k-2}\big)^2=\kappa_k,
$$
where $\calP_m$ denotes the polynomials of degree at most $m$ and $\calP_{-1}=\{0\}$. The first equality is the Gram determinant quotient, and the closed expression \eqref{kappa} is the classical Hilbert formula \cite{Choi1983,Szego1975}.

It remains to justify convergent expansions in $T^2$ despite the multiple zero eigenvalue at $T=0$. Lemma~\ref{lem_graded} applies to $\widehat{G}_T(A,\allowbreak c)$, with $(0,\allowbreak 1)$ replaced by $(-1/2,\allowbreak 1/2)$; denote the corresponding normalized matrices by $\widehat N$ and $\widehat R$. All the analyticity assertions then follow, and it only remains to observe that they hold in the variable $T^2$. Set $\widehat\varphi_j(T,\allowbreak s)=T^{-(j-1)}\langle q_j(A,\allowbreak c),\allowbreak e^{TsA}c\rangle$. Expanding the exponential gives $\widehat\varphi_j(-T,\allowbreak -s)=(-1)^{j-1}\widehat\varphi_j(T,\allowbreak s)$, and invariance of the scalar product of $L^2(-1/2,\allowbreak 1/2)$ under $s\mapsto-s$ gives $\widehat N(-T)=\Sigma\widehat N(T)\Sigma$, with $\Sigma=\diag((-1)^{i-1})$. Hence $({\bigwedge}^k\widehat N(-T))_{IJ}=(-1)^{\delta(I)+\delta(J)}({\bigwedge}^k\widehat N(T))_{IJ}$, while $D(-T)_{II}=(-1)^{\delta(I)}D(T)_{II}$. The two sign factors cancel, and therefore $\widehat R(-T)=\widehat R(T)$. It follows that $T^{-k^2}\widehat\lambda_1(T)\cdots\widehat\lambda_k(T)$ and the projectors $\widehat\Pi_k(T)$ are even analytic functions of $T$, hence analytic functions of $T^2$; the same holds for the ratios $T^{-(2k-1)}\widehat\lambda_k(T)$ and the oriented eigenvectors. Together with \eqref{centered_ratio}, this proves all the analytic assertions, uniformly on compact subsets of the set of controllable pairs.
\end{proof}

\subsection{Expansion of the transported Krylov data}

We now apply Proposition~\ref{prop_centered} to the analytic vector $b_T=e^{TA/2}b$ in \eqref{midpoint_identity}. Since $A$ commutes with $e^{TA/2}$, we have $\calK_k(A,\allowbreak b_T)=e^{TA/2}\calK_k(A,\allowbreak b)$.
The centered estimate first gives the more precise intermediate formulas
\begin{equation}\label{midpoint_estimates}
\lambda_k(T)=\kappa_kd_k(A,b_T)^2T^{2k-1}(1+\mathrm{O}(T^2)),\qquad \Pi_k(T)=\Pi_{\calK_k(A,b_T)}+\mathrm{O}(T^2).
\end{equation}

To expand the height, let $V_k=(b,\allowbreak Ab,\allowbreak \ldots,\allowbreak A^{k-1}b)\in\R^{n\times k}$ and let $\nu_k(T)$ be the Euclidean volume of the columns of $e^{TA/2}V_k$, with $\nu_0(T)=1$. Thus
$$
\nu_k(T)^2=\det(V_k^\top e^{TA^\top/2}e^{TA/2}V_k),\qquad d_k(A,b_T)^2=\frac{\nu_k(T)^2}{\nu_{k-1}(T)^2}.
$$
Differentiation of the Gram determinant at $T=0$ gives
$$
\frac{d}{dT}\log\nu_k(T)^2\vert_{T=0}=\sum_{j=1}^k\langle q_j,Aq_j\rangle.
$$
Taking the difference between levels $k$ and $k-1$ yields
\begin{equation}\label{height_expansion}
d_k(A,b_T)^2=d_k(A,b)^2(1+a_kT+\mathrm{O}(T^2)).
\end{equation}
Combining \eqref{midpoint_estimates} and \eqref{height_expansion} proves \eqref{eigenvalue_expansion}.

For the subspaces, let $E$ be fixed and let $E_T=e^{TA/2}E$. Differentiating an orthonormal frame of $E_T$ gives the standard projector identity
$$
\frac{d}{dT}\Pi_{E_T}\vert_{T=0}=\frac{1}{2}(\Pi_E^\perp A\Pi_E+\Pi_EA^\top\Pi_E^\perp).
$$
For $E=\calK_k(A,\allowbreak b)$ and $k\in\{1,\allowbreak \ldots,\allowbreak n-1\}$, the inclusion $A\calK_k(A,\allowbreak b)\subset\calK_{k+1}(A,\allowbreak b)$ and the construction of the orthonormal Krylov basis give
\begin{equation}\label{one_direction}
\Pi_{\calK_k(A,b)}^\perp A\Pi_{\calK_k(A,b)}=\beta_kq_{k+1}q_k^\top.
\end{equation}
Equations \eqref{midpoint_estimates}-\eqref{one_direction} prove \eqref{projector_expansion}. The rank-one projector onto $\psi_k(T)$ is $\Pi_k(T)-\Pi_{k-1}(T)$; differentiating this difference and using the chosen orientation gives \eqref{eigenvector_expansion}.

The analytic assertions for the original Gramian, and their joint character in $(T,\allowbreak A,\allowbreak b)$, follow from Lemma~\ref{lem_graded}. This completes the proof of Theorem~\ref{thm_main}.

\section{Spectral consequences and an exact energy identity}\label{sec_consequences}
We derive the Gaussian profile, the consequences for skew-symmetric and symmetric dynamics, and an exact forward Fokker--Planck energy identity with sharp directional smoothing estimates.

\subsection{Gaussian profile in moving principal coordinates}
Assume throughout this subsection that $(A,\allowbreak b)$ satisfies the Kalman condition. The transition density of the Ornstein-Uhlenbeck process at $y$ is then
$$
p_T(y)=\frac{1}{(2\pi)^{n/2}\det(G_T(A,b))^{1/2}}\exp\big(-\frac{1}{2}\langle G_T(A,b)^{-1}y,y\rangle\big).
$$
Set $\ell_k=\kappa_kd_k(A,\allowbreak b)^2$ and, for $z=(z_1,\allowbreak \ldots,\allowbreak z_n)\in\R^n$, follow the moving principal axes at their natural anisotropic scales:
$$
y_T(z)=\sum_{k=1}^n\ell_k^{1/2}T^{k-1/2}z_k\psi_k(T).
$$
These coordinates depend on $T$. The result below is therefore not a fixed-coordinate Taylor expansion of the transition density, but the profile observed at the natural Kalman scales in the moving eigenbasis of $G_T(A,\allowbreak b)$.

\begin{corollary}[refined Gaussian profile]
Uniformly for $z$ in compact subsets of $\R^n$,
\begin{equation}\label{density_expansion}
T^{n^2/2}\bigg(\prod_{k=1}^n\ell_k\bigg)^{1/2}p_T(y_T(z))=\frac{e^{-\Vert z\Vert^2/2}}{(2\pi)^{n/2}}\bigg(1+\frac{T}{2}\sum_{k=1}^na_k(z_k^2-1)+\mathrm{O}(T^2)\bigg).
\end{equation}
\end{corollary}

\begin{proof}
In the eigenbasis of $G_T(A,\allowbreak b)$, Theorem~\ref{thm_main} gives
$$
\langle G_T(A,b)^{-1}y_T(z),y_T(z)\rangle=\sum_{k=1}^nz_k^2(1-a_kT+\mathrm{O}(T^2)).
$$
It also gives $\det G_T(A,\allowbreak b)=(\prod_k\ell_k)T^{n^2}(1+T\sum_ka_k+\mathrm{O}(T^2))$, with $\sum_ka_k=\tr A$. Expanding the exponential and the square root proves \eqref{density_expansion}.
\end{proof}

At $z=0$, the correction is the known aggregate term $-\tr(A)T/2$ \cite{BarilariPaoli2017}, obtained from \eqref{determinant_expansion}. Away from zero, \eqref{density_expansion} resolves it into individual spectral contributions in moving coordinates; \eqref{eigenvector_expansion} describes their first rotation in the fixed Krylov basis.

\subsection{An even expansion for skew-symmetric dynamics}
When $A^\top=-A$, centering does more than simplify the proof: it removes every odd coefficient from the normalized eigenvalues.

\begin{corollary}
Assume that $A^\top=-A$ and that $(A,\allowbreak b)$ satisfies the Kalman condition. Then
\begin{equation}\label{skew_centering}
G_T(A,b)=e^{TA/2}\widehat{G}_T(A,b)e^{-TA/2}.
\end{equation}
Consequently, for every $k\in\{1,\allowbreak \ldots,\allowbreak n\}$, the function $T^{-(2k-1)}\lambda_k(T)$ extends near $T=0$ to a real-analytic function of $T^2$. In particular, as $T\to0^+$,
$$
\lambda_k(T)=\kappa_kd_k(A,b)^2T^{2k-1}(1+\mathrm{O}(T^2)),
$$
and every odd relative coefficient vanishes. Moreover, for $k\in\{1,\allowbreak \ldots,\allowbreak n-1\}$, as $T\to0^+$,
$$
\Pi_k(T)=e^{TA/2}\Pi_{\calK_k(A,b)}e^{-TA/2}+\mathrm{O}(T^2).
$$
\end{corollary}

\begin{proof}
Writing $t=s+T/2$ gives $e^{tA}=e^{TA/2}e^{sA}$. Since $e^{TA/2}$ is orthogonal, its transpose is $e^{-TA/2}$, and the centered identity becomes \eqref{skew_centering}. Orthogonal conjugation preserves eigenvalues and transports spectral projectors. Proposition~\ref{prop_centered}, applied with the fixed vector $b$, therefore gives all the assertions.
\end{proof}

Here $a_k=\langle q_k,\allowbreak Aq_k\rangle=0$. The projectors themselves need not be even in fixed coordinates, because they undergo the explicit orthogonal conjugation in \eqref{skew_centering}. This parity is also relevant to the resonant actuator-design problem studied in \cite{TrelatZuazua2026actuator}.

\subsection{An inverse consequence for symmetric dynamics}

For symmetric dynamics, eigenvalue coefficients determine a Jacobi representative, while ambient projector data restore the original coordinates.

Indeed, when $A=A^\top$, its matrix in the Krylov basis is symmetric tridiagonal, with diagonal $a_k$ and adjacent entries $\beta_k$.

\begin{proposition}
Assume that $A=A^\top$ and that $(A,\allowbreak b)$ satisfies the Kalman condition. Define the first two normalized eigenvalue coefficients by
$$
\ell_k=\lim_{T\to0^+}T^{-(2k-1)}\lambda_k(T),\qquad r_k=\lim_{T\to0^+}\frac{1}{T}\big(\frac{\lambda_k(T)}{\ell_kT^{2k-1}}-1\big).
$$
Then $r_k=a_k$ and
\begin{equation}\label{inverse_adjacent}
\beta_k=\sqrt{\frac{\kappa_k\ell_{k+1}}{\kappa_{k+1}\ell_k}},\qquad k\in\{1,\ldots,n-1\}.
\end{equation}
Consequently, the scalar coefficients $(\ell_k,\allowbreak r_k)_{1\leq k\leq n}$ determine the Jacobi representation of the pair $(A,\allowbreak b)$ up to an orthogonal change of coordinates and the unavoidable global sign of $b$.

If the projectors are observed in the original coordinates, set
$$
\Pi_k^0=\lim_{T\to0^+}\Pi_k(T),\qquad \dot\Pi_k^0=\lim_{T\to0^+}\frac{\Pi_k(T)-\Pi_k^0}{T},\qquad \Pi_0^0=0,\qquad \Pi_n^0=\Id.
$$
Then the matrix $A$ itself is reconstructed by
\begin{equation}\label{inverse_matrix}
A=\sum_{k=1}^nr_k(\Pi_k^0-\Pi_{k-1}^0)+2\sum_{k=1}^{n-1}\dot\Pi_k^0.
\end{equation}
Moreover, $\Vert b\Vert=\sqrt{\ell_1}$ and $\Pi_1^0$ determine $b$ up to its global sign, which cannot be recovered from the Gramian.
\end{proposition}

\begin{proof}
Theorem~\ref{thm_main} gives $\ell_k=\kappa_kd_k(A,\allowbreak b)^2$ and $r_k=a_k$. Since $\beta_k=d_{k+1}(A,\allowbreak b)/d_k(A,\allowbreak b)$, formula \eqref{inverse_adjacent} follows. Symmetry and $A\calK_k(A,\allowbreak b)\subset\calK_{k+1}(A,\allowbreak b)$ imply that the only nonzero off-diagonal entries of $A$ in the basis $(q_1,\allowbreak \ldots,\allowbreak q_n)$ are the adjacent entries $\beta_k$. Thus the coefficients determine the Jacobi matrix with diagonal $(r_k)$ and positive adjacent entries $(\beta_k)$, while $b=\sqrt{\ell_1}e_1$ in these model coordinates. This proves the first assertion.

For the reconstruction in the original coordinates, \eqref{projector_expansion} gives $2\dot\Pi_k^0=\beta_k(q_{k+1}q_k^\top+q_kq_{k+1}^\top)$, and summing the diagonal and adjacent parts gives \eqref{inverse_matrix}. Finally, $\Pi_1^0$ determines only the line $\R q_1$, and $\Vert b\Vert=\sqrt{\ell_1}$ because $\kappa_1=1$. Hence $b$ is reconstructed up to its global sign. This ambiguity is unavoidable, since $G_T(A,\allowbreak b)=G_T(A,\allowbreak -b)$, and it is the only one.
\end{proof}

The moment--orthogonal-polynomial--Jacobi correspondence and its Lanczos realization are classical \cite{GolubMeurant2010}; here the spectral jets supply the recurrence coefficients directly. Eigenvalue coefficients identify the pair up to orthogonal coordinates, and ambient projectors recover those coordinates. For nonsymmetric $A$, the same data give $a_k$ and $\beta_k$ but not the remaining Krylov coefficients, so the first jet does not suffice.

\subsection{An exact energy identity and sharp directional smoothing for the forward equation}
The identity below does not require the Kalman condition.
The Lyapunov equation also gives a time-dependent energy identity for the forward Fokker-Planck equation associated with $dX_t=AX_t\,\allowbreak dt+b\,\allowbreak dW_t$,
\begin{equation}\label{fokker_planck}
\partial_t\rho=-\operatorname{div}(Ax\,\rho)+\frac{1}{2}(b\cdot\nabla)^2\rho.
\end{equation}
The distinction between forward and backward equations is essential. For the backward generator $\calL$ used in the introduction, the analogous multiplier is $G_t(-A,\allowbreak b)$; for the adjoint equation \eqref{fokker_planck}, it is $G_t(A,\allowbreak b)$. For $\theta\geq0$, set
$$
\calE_\theta(t)=\theta\Vert\rho(t)\Vert_{L^2}^2+\int_{\R^n}(\nabla\rho(t))^\top G_t(A,b)\nabla\rho(t)\,dx.
$$

\begin{proposition}\label{prop_energy}
Let $\rho_0$ belong to the Schwartz class and let $\rho$ solve \eqref{fokker_planck}. Then, for every $\theta\geq0$,
\begin{equation}\label{energy_identity}
\begin{aligned}
\frac{d}{dt}\big(e^{t\tr A}\calE_\theta(t)\big)
=-e^{t\tr A}\bigg(& (\theta-1)\Vert(b\cdot\nabla)\rho\Vert_{L^2}^2\\
&+\int_{\R^n}\big((b\cdot\nabla)\nabla\rho\big)^\top G_t(A,b)\big((b\cdot\nabla)\nabla\rho\big)\,dx\bigg).
\end{aligned}
\end{equation}
Consequently, $e^{t\tr A}\calE_\theta(t)$ is nonincreasing for every initial datum when $\theta\geq1$. If $b\ne0$, then $\theta=1$ is the smallest coefficient with this universal monotonicity property.
\end{proposition}

\begin{proof}
Set $D_b=b\cdot\nabla$. Integration by parts in \eqref{fokker_planck} gives
$$
\frac{d}{dt}\big(e^{t\tr A}\Vert\rho\Vert_{L^2}^2\big)=-e^{t\tr A}\Vert D_b\rho\Vert_{L^2}^2.
$$
Moreover, differentiating the equation gives
$$
\partial_t\nabla\rho=-A^\top\nabla\rho-(Ax)\cdot\nabla(\nabla\rho)-(\tr A)\nabla\rho+\frac{1}{2}D_b^2\nabla\rho.
$$
After integration by parts, the Lyapunov identity $\partial_tG_t=AG_t+G_tA^\top+bb^\top$ cancels the drift contributions, while the term $bb^\top$ produces $\Vert D_b\rho\Vert_{L^2}^2$. Hence
$$
\frac{d}{dt}\big(e^{t\tr A}\calE_0(t)\big)=e^{t\tr A}\bigg(\Vert D_b\rho\Vert_{L^2}^2-\int_{\R^n}(D_b\nabla\rho)^\top G_t(A,b)(D_b\nabla\rho)\,dx\bigg).
$$
Multiplying the first identity by $\theta$ and adding the second one proves \eqref{energy_identity}. If $0\leq\theta<1$ and $b\ne0$, choose $\rho_0$ such that $D_b\rho_0\ne0$. Since $G_0=0$, the derivative at $t=0$ then equals $(1-\theta)\Vert D_b\rho_0\Vert_{L^2}^2>0$, proving the last assertion. All integrations are justified for Schwartz data; the identity then extends by standard approximation whenever the displayed quantities are finite.
\end{proof}

Identity \eqref{energy_identity} is exact at every time for which the solution and the displayed integrals are defined. Only the layerwise analysis below uses small time. No stability assumption on $A$ or invariant measure is required. The identity is related to time-dependent corrected energies used in hypocoercivity \cite{Herau2007,Villani2009}, but does not by itself assert convergence to equilibrium. When $(A,\allowbreak b)$ is controllable and $0<t<T_0$, its gradient term has the exact spectral decomposition
$$
\int_{\R^n}(\nabla f)^\top G_t(A,b)\nabla f\,dx=\sum_{k=1}^n\lambda_k(t)\Vert\psi_k(t)\cdot\nabla f\Vert_{L^2}^2.
$$
Recall that $Y_k=q_k\cdot\nabla$ and $\beta_0=\beta_n=0$. For a fixed $f$ in the Schwartz class, set
$$
J_k(f)=a_k\Vert Y_kf\Vert_{L^2}^2+\beta_k\langle Y_kf,Y_{k+1}f\rangle_{L^2}-\beta_{k-1}\langle Y_kf,Y_{k-1}f\rangle_{L^2},
$$
where the nonexistent end terms are omitted. Theorem~\ref{thm_main} gives, layer by layer,
$$
\frac{\lambda_k(t)\Vert\psi_k(t)\cdot\nabla f\Vert_{L^2}^2}{\kappa_kd_k(A,b)^2t^{2k-1}}=\Vert Y_kf\Vert_{L^2}^2+tJ_k(f)+\mathrm{O}(t^2)\Vert\nabla f\Vert_{L^2}^2.
$$
Thus $a_k$ corrects the diagonal energy of the $k$-th layer, while the adjacent Krylov ratios couple it to its neighboring layers.

\begin{remark}[Layerwise truncation]\label{rem_layerwise}
The preceding expansions must be truncated relative to each Kalman layer. Since their leading scales are $t^{2k-1}$, terms must not be pooled merely because they have the same absolute power of $t$: the remainder in one layer can have the order of the leading term in the next one. For the double integrator, the order-$t$ rotation of the large spectral branch contributes $t^3/4$ in the position direction, while the leading contribution of the small branch is $t^3/12$; together they recover the exact coefficient $t^3/3$.
\end{remark}

For the double integrator of Example~\ref{ex_integrator}, equation \eqref{fokker_planck} is $\partial_t\rho+v\partial_x\rho=\frac{1}{2}\partial_v^2\rho$ and
$$
\calE_1(t)=\Vert\rho\Vert_{L^2}^2+t\Vert\partial_v\rho\Vert_{L^2}^2+t^2\langle\partial_x\rho,\partial_v\rho\rangle_{L^2}+\frac{t^3}{3}\Vert\partial_x\rho\Vert_{L^2}^2.
$$
The effective velocity coefficient is $(e_v^\top G_t^{-1}e_v)^{-1}=t/4$; the position coefficient is $(e_x^\top G_t^{-1}e_x)^{-1}=t^3/12$.
Monotonicity consequently gives, for instance,
$$
\Vert\partial_v\rho(t)\Vert_{L^2}\leq2t^{-1/2}\Vert\rho_0\Vert_{L^2},\qquad \Vert\partial_x\rho(t)\Vert_{L^2}\leq2\sqrt3\,t^{-3/2}\Vert\rho_0\Vert_{L^2}.
$$
The powers $t,\allowbreak t^2,\allowbreak t^3$ match the familiar short-time weights in corrected energies for the free Kolmogorov model. In general, the matrix Cauchy-Schwarz inequality gives
$$
\vert q_k^\top\nabla\rho\vert^2\leq(q_k^\top G_t(A,b)^{-1}q_k)(\nabla\rho)^\top G_t(A,b)\nabla\rho.
$$
Lemma~\ref{lem_graded} gives $q_k^\top G_t(A,\allowbreak b)^{-1}q_k=t^{-(2k-1)}(N(t,\allowbreak A,\allowbreak b)^{-1})_{kk}$, so that the monotonicity of $e^{t\tr A}\calE_1$ yields, for every $k\in\{1,\allowbreak \ldots,\allowbreak n\}$,
$$
\Vert Y_k\rho(t)\Vert_{L^2}\leq\big((N(t,A,b)^{-1})_{kk}\big)^{1/2}\,t^{-k+1/2}\,e^{-t\tr(A)/2}\,\Vert\rho_0\Vert_{L^2},
$$
the constants converging to those given by the moment matrix $N(0,\allowbreak A,\allowbreak b)$ of Lemma~\ref{lem_graded}. For the double integrator, $N$ does not depend on $t$, with $(N^{-1})_{11}=4$ and $(N^{-1})_{22}=12$, which are exactly the two constants above. These constants follow directly from the energy identity, but are not the exact operator norms. The exact norm and its sharp small-time consequences are as follows.

\begin{corollary}[sharp directional smoothing]
Assume that $(A,\allowbreak b)$ satisfies the Kalman condition, and let $S_t^*$ denote the forward evolution in \eqref{fokker_planck}. For every unit vector $q$ and every $t>0$,
\begin{equation}\label{exact_smoothing_norm}
\Vert(q\cdot\nabla)S_t^*\Vert_{L^2\to L^2}=e^{-1/2}e^{-t\tr(A)/2}(q^\top G_t(A,b)^{-1}q)^{1/2}.
\end{equation}
Consequently, for every $k\in\{1,\allowbreak \ldots,\allowbreak n\}$, as $t\to0^+$, one has, along the moving principal directions,
\begin{equation}\label{moving_smoothing}
\Vert(\psi_k(t)\cdot\nabla)S_t^*\Vert_{L^2\to L^2}=\frac{e^{-1/2}}{\kappa_k^{1/2}d_k(A,b)}t^{-k+1/2}\big(1-\frac{a_k+\tr A}{2}t+\mathrm{O}(t^2)\big),
\end{equation}
whereas in the fixed Krylov directions,
\begin{equation}\label{sharp_krylov_smoothing}
\Vert Y_kS_t^*\Vert_{L^2\to L^2}=e^{-1/2}\frac{((\mathfrak H_n^{-1})_{kk})^{1/2}}{d_k(A,b)}t^{-k+1/2}(1+\mathrm{O}(t)).
\end{equation}
The universal matrix entry in \eqref{sharp_krylov_smoothing} is explicitly
$$
(\mathfrak H_n^{-1})_{kk}=\frac{((n+k-1)!)^2}{(2k-1)((n-k)!)^2((k-1)!)^2}.
$$
\end{corollary}

\begin{proof}
In Fourier variables,
$$
\widehat\rho(t,\xi)=e^{-\frac{1}{2}\xi^\top G_t(A,b)\xi}\widehat\rho_0(e^{tA^\top}\xi).
$$
Using Plancherel's identity, the change of variables $\eta=e^{tA^\top}\xi$ and the identity $e^{-tA}G_t(A,\allowbreak b)e^{-tA^\top}=G_t(-A,\allowbreak b)$, the operator norm in \eqref{exact_smoothing_norm} is $e^{-t\tr(A)/2}$ times
$$
\sup_{\eta\in\R^n}\vert\langle e^{-tA}q,\eta\rangle\vert e^{-\frac{1}{2}\eta^\top G_t(-A,b)\eta}.
$$
Writing $\eta=G_t(-A,\allowbreak b)^{-1/2}w$ and using $\sup_{r\geq0}re^{-r^2/2}=e^{-1/2}$ reduces this supremum to
$$
e^{-1/2}((e^{-tA}q)^\top G_t(-A,b)^{-1}(e^{-tA}q))^{1/2}.
$$
The quadratic form inside the square root equals $q^\top G_t(A,\allowbreak b)^{-1}q$, which proves \eqref{exact_smoothing_norm}.

Taking $q=\psi_k(t)$ in \eqref{exact_smoothing_norm} and using $\psi_k(t)^\top G_t(A,\allowbreak b)^{-1}\psi_k(t)=\lambda_k(t)^{-1}$ together with \eqref{eigenvalue_expansion} gives \eqref{moving_smoothing}. For the fixed direction $q_k$, equations \eqref{graded_factorization} and \eqref{moment_factorization} give
$$
q_k^\top G_t(A,b)^{-1}q_k=t^{-(2k-1)}(N(t,A,b)^{-1})_{kk}=\frac{(\mathfrak H_n^{-1})_{kk}}{d_k(A,b)^2}t^{-(2k-1)}(1+\mathrm{O}(t)).
$$
This proves \eqref{sharp_krylov_smoothing}. Let $D_f=\diag(1/0!,\allowbreak \ldots,\allowbreak 1/(n-1)!)$. The scaled and standard Hilbert matrices are related by
$\mathfrak H_n=D_f\big(1/(i+j-1)\big)_{1\leq i,\allowbreak j\leq n}D_f$.
The closed formula therefore follows by applying the classical inverse Hilbert formula \cite{Choi1983} to the middle matrix.
\end{proof}

Since $q^\top G_t(A,\allowbreak b)^{-1}q$ is the minimum energy needed to steer the origin to $q$ in time $t$, formula \eqref{exact_smoothing_norm} identifies the directional smoothing norm with $e^{-1/2}e^{-t\tr(A)/2}$ times the square root of the control cost in the same direction. Taking the supremum over unit directions gives the exact identity
$$
\sup_{\Vert y\Vert=1}\inf_{\substack{u\in L^2(0,t)\\x(0)=0,\ x(t)=y}}\int_0^t\vert u(\tau)\vert^2d\tau=e^{1+t\tr(A)}\bigg(\sup_{\Vert q\Vert=1}\Vert(q\cdot\nabla)S_t^*\Vert_{L^2\to L^2}\bigg)^2.
$$
For $0<t<T_0$, both suprema are attained in the weakest spectral direction $q=y=\psi_n(t)$, and \eqref{control_energy} is the corresponding small-time expansion of the left-hand side. This exact dictionary explains why the smoothing blow-up exponents are one half of the corresponding control-energy exponents.

The convergence $\psi_k(t)\to q_k$ does not make the two leading constants coincide, because the inverse-Gramian metric becomes singular as $t\to0^+$. Indeed, the ratio of the fixed-direction leading norm constant in \eqref{sharp_krylov_smoothing} to the moving-direction one in \eqref{moving_smoothing} is
$$
\big(\kappa_k(\mathfrak H_n^{-1})_{kk}\big)^{1/2}=\binom{n+k-1}{n-k}.
$$
It equals $n$ for $k=1$ and $1$ precisely when $k=n$; it is larger than $1$ for every $k<n$. In the latter case, the vanishing components of $q_k$ along the weaker moving spectral branches are amplified by the inverse eigenvalues and remain visible in the leading coefficient; this is the inverse-Gramian counterpart of the layerwise truncation in Remark~\ref{rem_layerwise}. In the last layer no weaker branch is present, consistently with $(\mathfrak H_n^{-1})_{nn}=1/\kappa_n$.

Comparison of \eqref{sharp_krylov_smoothing} with the bound obtained above from the monotonicity of $e^{t\tr A}\calE_1$ shows that the powers $t^{-k+1/2}$ are optimal and that the constants produced by monotonicity are larger by the factor $e^{1/2}$. Related matrix-valued constructions in \cite{ArnoldErb} contain diagonal and adjacent monomials that also occur in the Taylor expansion of $G_t$, with coefficients selected through a finite absorption argument. We record \eqref{energy_identity} as an exact consequence of the linear Lyapunov structure and as a useful interpretation of the spectral coefficients, not as a general hypocoercivity principle.

\begin{example}[damped Langevin dynamics]
Let
$$
A=\begin{pmatrix}0&1\\-\omega^2&-\gamma\end{pmatrix},\qquad b=e_2,
$$
where $\omega>0$ and $\gamma\geq0$. The orthonormal Krylov basis is $(e_2,\allowbreak e_1)$, and $a_1=-\gamma$, $a_2=0$, $\beta_1=1$. Thus
$$
\lambda_1(T)=T(1-\gamma T+\mathrm{O}(T^2)),\qquad \lambda_2(T)=\frac{T^3}{12}(1+\mathrm{O}(T^2)),
$$
as $T\to 0^+$, while the two principal directions have the same first-order rotation as in Example~\ref{ex_integrator}. Since $a_1=-\gamma$ and $a_2=0$, damping is detected at first order by the large covariance axis but not by the small one, whereas the frequency $\omega$ enters only at higher order.
\end{example}

\section{Conclusion and perspectives}\label{sec_conclusion}
Theorem~\ref{thm_main} refines the known spectral powers by explicit first corrections to eigenvalues and eigendirections, with analytic normalized data. These coefficients refine control energy, Gaussian profiles and directional smoothing, and reconstruct symmetric dynamics: $a_k$ describes the action within each Krylov direction and $\beta_k$ the transfer to the next layer.

\smallskip
\paragraph{Higher-order coefficients}
Theorem~\ref{thm_main} gives convergent expansions whose higher coefficients can be computed from the orbit Taylor series and analytic perturbation equations; \eqref{minor_identity} and Remark~\ref{rem_minors} give a route through Gram determinants. At relative order $T^2$, nonadjacent Krylov coefficients, the even centered correction and the discarded second-order perturbation terms interact. We do not know a closed invariant formula as simple as \eqref{eigenvalue_expansion}--\eqref{eigenvector_expansion}. The first-order term is special: for fixed $c$, $T^{-1}\widehat G_T(A,c)$ and its normalized spectral data are even in $T$, so that term arises entirely from $b\mapsto e^{TA/2}b$.

\smallskip
\paragraph{Addition of a scalar drift}
Replacing $A$ by $A+\zeta\Id$, with $\zeta\in\R$, leaves all Krylov spaces and all coefficients $\beta_k$ unchanged and replaces every $a_k$ by $a_k+\zeta$. Hence each eigenvalue acquires the same relative factor $1+\zeta T+\mathrm{O}(T^2)$, while the first variation of the spectral flag is unchanged. This is consistent with the scalar factor $e^{\zeta t}$ in the orbit.

\smallskip
\paragraph{Multi-input case}
For a controllable system $\dot x=Ax+Bu$ with $A\in\R^{n\times n}$ and $B\in\R^{n\times m}$, set $\calF_j=\operatorname{Ran}(B,\allowbreak AB,\allowbreak \ldots,\allowbreak A^jB)$ for $j\geq0$, with $\calF_{-1}=\{0\}$. The dimensions $(\dim\calF_j)_j$ form the growth vector; equivalently, the increments $m_j=\dim\calF_j-\dim\calF_{j-1}$ determine the multiplicities of the successive layers. The leading eigenvalues form blocks: the exponent $2j+1$ occurs with multiplicity $m_j$. This leading block structure is classical and fits the graded perturbation framework of \cite{UsevichBarthelme2026}. Extending Theorem~\ref{thm_main} is substantially more delicate: the scalar heights are replaced by positive forms on quotient spaces, eigenvalues inside a layer need not be simple, and their first projectors may depend on internal spectral splittings. A coordinate-free two-term block formula is therefore a separate problem, rather than a routine extension of the present proof.

\smallskip
\paragraph{Growing dimension}
The argument does not control a joint limit $T\to0^+$ and $n\to+\infty$: both $\kappa_n$ and the last Krylov height may deteriorate with dimension, and compact-family uniformity gives no bound uniform in $n$. Such a joint limit would require observability and fast-control estimates uniform in the discretization.

\begingroup
\footnotesize
\section*{Acknowledgements}
The second author was partially supported by the European Research Council (ERC) through the European Union's Horizon Europe programme (ERC Advanced Grant CoDeFeL, grant agreement No.~101096251); the Air Force Office of Scientific Research under award No.~FA8655-24-1-7027; the Alexander von Humboldt Professorship; the European Union's Horizon Europe MSCA Doctoral Network ModConFlex (grant agreement No.~101073558); the Research Council of Norway through SURE-AI, grant No.~357482; and Grant PID2023-146872OB-I00 (DyCMaMod), funded by MICIU/AEI/10.13039/501100011033 and by ERDF/EU\@. This article is based upon work from COST Actions CA24122 (mSPACE) and CA24136 (InterCoML), supported by COST (European Cooperation in Science and Technology). 
\par\smallskip
\noindent
AI tools (ChatGPT and Claude) were used solely to improve the exposition, check the bibliography, and compare versions. All scientific ideas, results, proofs, and initial drafts are the authors' own. The authors assume responsibility for all content.
\par
\endgroup

\end{document}